\documentclass{daj}

\dajAUTHORdetails{%
  title = {Random Volumes in $d$-dimensional Polytopes}, 
  author = {Alan Frieze, Wesley Pegden, and Tomasz Tkocz},
  plaintextauthor = {Alan Frieze, Wesley Pegden, Tomasz Tkocz},
  plaintexttitle = {Random Volumes in d-dimensional Polytopes}, 
  keywords = {random polytope, high dimensional convex body, simplex, simplicial polytope, volume threshold},
}   

\dajEDITORdetails{%
   year={2020},
   number={15},
   received={28 February 2020},   
   published={22 September 2020},  
   doi={10.19086/da.17109},       
}   

\usepackage{diagbox}
\usepackage{mathtools}
\usepackage{bbm}
\usepackage{latexsym}
\usepackage{epsfig}
\usepackage{amsmath,amsthm,amssymb,enumerate,hyperref,bm}
\usepackage[active]{srcltx}
\usepackage{color}

\def\cI{{\cal I}}

\def\a{\alpha}   
\def\e{\varepsilon}    \def\g{\gamma}
  
 \def\th{\theta}    \def\l{\lambda}
 \def\m{\mu}  
\def\r{\rho}

\def\cD{{\cal D}}
\def\cP{{\cal P}}

\def\cC{{\cal C}}
\def\cB{{\cal B}}

\def\cM{{\cal M}}

\newtheorem{theorem}{Theorem}[section]

\newtheorem{lemma}[theorem]{Lemma}

\theoremstyle{definition}
\newtheorem{Remark}[theorem]{Remark}
\newtheorem{definition}{Definition}

\def\cA{{\mathcal A}}
\def\cB{{\mathcal B}}
\def\cC{{\mathcal C}}
\def\cD{{\mathcal D}}
\def\cR{{\mathcal R}}

\newcommand{\Vol}{\mathrm{Vol}}

\newcommand{\brac}[1]{\left(#1\right)}

\newcommand{\bfrac}[2]{\left(\frac{#1}{#2}\right)}

\def\cE{{\cal E}}

\newcommand{\set}[1]{\left\{#1\right\}}

\def\E{\mathbb{E}}

\def\Pr{\mathbb{P}}

\newcommand{\ignore}[1]{}

\newcommand{\beq}[2]{\begin{equation}\label{#1}#2\end{equation}}

\def\cI{{\cal I}}

\newcommand{\mults}[1]{\begin{multline*}#1\end{multline*}}

\usepackage{tikz}
\usetikzlibrary{arrows}
\usetikzlibrary{decorations}
\usetikzlibrary{shapes.misc}

\def\bl{{\bm \lambda}}
\def\bx{{\bf x}}
\def\by{{\bf y}}
\def\bp{{\bf p}}
\def\bq{{\bf q}}

\def\hE{\widehat{E}}

\def\vol{\mathrm{vol}}
\def\be{{\boldsymbol{\e}}}

\newcommand{\1}{\textbf{1}}
\newcommand{\pp}[1]{\mathbb{P}\left( #1 \right)}
\newcommand{\scal}[2]{\left\langle #1, #2 \right\rangle}
\DeclareMathOperator{\conv}{conv}

\begin{document}

\begin{frontmatter}[classification=text]

\title{Random Volumes in $d$-dimensional Polytopes} 

\author[af]{Alan Frieze\thanks{Supported in part by NSF grant DMS1952285}}
\author[wp]{Wesley Pegden\thanks{Supported in part by NSF grant DMS1363136}}
\author[t]{Tomasz Tkocz\thanks{Supported in part by NSF grant DMS1955175}}

\begin{abstract}
Suppose we choose $N$ points uniformly randomly from a convex body in $d$ dimensions.  How large must $N$ be, asymptotically with respect to $d$, so that the convex hull of the points is nearly as large as the convex body itself?  It was shown by Dyer-F\"uredi-McDiarmid that exponentially many samples suffice when the convex body is the hypercube, and by Pivovarov that the Euclidean ball demands roughly $d^{d/2}$ samples.  We show that when the convex body is the simplex, exponentially many samples suffice; this then implies the same result for any convex simplicial polytope with at most exponentially many faces.
\end{abstract}
\end{frontmatter}


\section{Introduction}
Suppose that points $\bq_1,\bq_2,\dots$ are sampled uniformly and independently from a convex body $X\subseteq \Re^d$. We are interested in the asymptotics of the random variable $V_{X,N}$ given by the volume of the convex hull of $\bq_1,\dots,\bq_N$.  In particular, we would like to know how large $N$ has to be to ensure that w.h.p.\footnote{with high probability: probability approaching 1 as $d\to \infty$} the volume of the convex hull of $\bq_1,\dots,\bq_N$ is a significant fraction of the volume of $X$.

This problem is well understood when $X$ is a product space (i.e., a hypercube) or a Euclidean ball.  In the case where $X$ is the hypercube $[0,1]^d$, the coordinates of the $\bq_i$ are independent uniform random variables in $[0,1]$, and Dyer, F\"uredi, and McDiarmid \cite{DFM} proved the following theorem.
\begin{theorem}[Dyer, F\"uredi, McDiarmid, 1992]\label{t.hypercube}
  If $X$ is the hypercube $[0,1]^d$, $\l=e^{\int_{0}^\infty\brac{\frac 1 u - \frac 1 {e^u-1}}^2du}\approx 2.14$, and $\e>0$, then as $d\to \infty$ we have that
  \[
  \E V_{X,N}\to\begin{cases}0& \text{if }N=N(d)\leq \brac{\l-\e}^d,\\
  1& \text{if }N=N(d)> \brac{\l+\e}^d.
  \end{cases}
  \]
\end{theorem}
In particular, an exponential number of sample points suffice to capture the volume of the hypercube with the convex hull of the sample (and they even determine the correct base of the exponent).  This was generalized in 2009 by Gatzouras and Giannopoulos in \cite{GGindep} to the case of random points with i.i.d.~coordinates which instead of being uniform are drawn from any even, compactly supported distribution that satisfies certain mild conditions.

On the other hand, if $X$ is the Euclidean ball, Pivovarov proved in \cite{Piv} that the threshold is super-exponential.

\begin{theorem}[Pivovarov, 2007]\label{t.ball}
  If $X$ is the unit Eulidean ball in $\Re^d$, $X = \{x \in \Re^d, \ \sum_{i=1}^d x_i^2 \leq 1\}$, and $\e>0$, then as $d\to \infty$ we have that
  \[
  \frac{\E V_{X,N}}{\Vol(X)}  \to\begin{cases}0& \text{if }N=N(d)\leq d^{\frac{d}{2}(1-\e)},\\
  1& \text{if }N=N(d)> d^{\frac{d}{2}(1+\e)}.
  \end{cases}
  \]
\end{theorem}
For results concerning a more general rotationally symmetric model of the so-called $\beta$-polytopes (also exhibiting super-exponential thresholds), see the recent papers \cite{B1,B2}. For general bounds on $N$ concerning arbitrary log-concave and $\kappa$-concave distributions see \cite{kappa}.

We analyze the case where $X$ is a convex simplicial polytope: that is, a polytope whose facets are all simplices. In particular, we prove the following result.
  \begin{theorem}\label{t.polytope}
    Let $X\subseteq \Re^d$ be a convex simplicial polytope with $m$ facets, let $\bq_1,\bq_2,\dots$ be a sequence of points chosen independently and uniformly from $X$, and let $Q_j=Q_{j,d}\subseteq \Omega$ be the convex hull of $\{\bq_1,\dots,\bq_j\}$. There are positive universal constants $c_0, C_0$ such that if $d$ is sufficiently large and $N>C_0^dm$, then $\Vol(Q_{N}) \geq (1-e^{-c_0\sqrt{d}} )\Vol(X)$.
  \end{theorem}
  
  Since any convex simplicial polytope with $m$ faces can be partitioned into at most $m$ simplices, which are all affine equivalent, it suffices to prove Theorem \ref{t.polytope} in the case where $X$ is a simplex.  In particular, we let $\Omega_d$ denote the standard embedding of the $(d-1)$-dimensional simplex in $d$-dimensional space:
  \[
  \Omega_d=\set{\bx\geq {\bf 0}:x_1+x_2+\cdots+x_d=1}.
  \]
  The heart of our results is thus the following statement.
  \begin{theorem}\label{t.coveromega}
    Let $\bq_1,\bq_2,\dots$ be a sequence of points chosen independently and uniformly from $\Omega=\Omega_d$, and let $Q_j=Q_{j,d}\subseteq \Omega$ be the convex hull of $\{\bq_1,\dots,\bq_j\}$. There are positive constants $c_0, C_0$ such that if $d$ is sufficiently large and $N>C_0^d$, then $\E\Vol(Q_{N})\geq (1-e^{-c_0\sqrt{d}})\Vol(\Omega)$.
  \end{theorem}

\begin{Remark}\label{rem:conv-a.s.}
By the Borel-Cantelli lemma, it follows that if we take a sequence of instances $\Omega_1,\Omega_2,\dots,$ then $\Vol(Q_{N,d})/\Vol(\Omega_d)\to 1$ as $d\to \infty$ with probability 1.
\end{Remark}

\begin{Remark}\label{rem:consts-final}
For clarity, we do not try to optimize any constants in our proofs. We get the theorem with $c_0 = \frac14$ and $C_0 = 300$.
\end{Remark} 

The following lower bound shows that an exponential dependence is necessary.
  \begin{theorem}\label{t.lower}
Under the assumptions of Theorem \ref{t.coveromega}, for every $\e > 0$, if $N < e^{(\gamma-\e)d}$, then we have $\frac{1}{\Vol(\Omega_d)}\E \Vol(Q_N) \to 0$ as $d\to\infty$, where $\gamma = 0.577\ldots$ is the Euler-Mascheroni constant.
  \end{theorem}

A similar lower bound with a worse constant follows from Theorem 1 in \cite{kappa}.  To prove Theorem \ref{t.lower}, we use the approach from \cite{DFM}. We conjecture that the value of the constant $e^{\gamma}$ is sharp. (The method from \cite{DFM} yields sharp results in the independent case as well as rotationally symmetric ones -- see \cite{B1, B2, DFM, Piv} -- where the dependence between components is mild, as in the case of a simplex.) For the upper bound, we follow a different strategy, which is summarized at the beginning of the next section.

The rest of the paper comprises two sections, which are devoted to the proofs of Theorems \ref{t.coveromega} and~\ref{t.lower}.

  \section{Proof of the upper bound: Theorem \ref{t.coveromega}}
  
We begin by sketching the structure of the whole proof.      For $i=1,\dots,d$ we define the \emph{$\alpha$-caps} $C_i(\alpha)$ of the simplex to be the sets
\begin{equation}\label{e.cap}
  C_i(\alpha):=\Omega\cap \set{\bx\mid x_i\geq 1-\alpha}.
\end{equation}
Note that they are disjoint as long as $\a < \frac12$, and the volume $\Vol(C_i(\alpha))$ of $C_i(\alpha)$ is precisely $\alpha^{d-1}\cdot \Vol(\Omega)$.  In particular, when examining the sequence $\{\bq_j\}$, we expect to see a point in $C_{i}(\alpha)$ 
every $(\tfrac 1 \alpha)^{d-1}$ steps.  And for $\alpha$ a constant, after exponentially many steps, we can collect points from each cap $C_i(\alpha)$.  A routine calculation shows that the expected measure of the convex hull of a random set of $d$ points with one from each $C_i(\alpha)$ is exponentially small compared with $\Omega$, though it is not a priori clear how much overlap to expect from multiple such random simplices.  The basic strategy of the proof is to define a large set $\Omega(\be,\g)\subseteq \Omega$, and then show that for any \emph{fixed} $\bx\in \Omega(\be,\g)$, the point $\bx$ is very likely to lie in the convex hull of some simplex with one point $\bp^\bx_{i}$ in each in cap $C_i(\alpha)$, where all the the points $\bp^\bx_1,\bp^\bx_1,\dots,\bp^\bx_d$ occur among the first $C_0^d$ terms of the sequence $\bq_1,\bq_2,\dots$.  We do this by showing (in Lemma \ref{l.goodpi}) that every exponentially many steps, one obtains not only a point $\bp^\bx_i$ which lies in the cap $C_i(\a)$, but one which is similar to $\bx$ with respect to its proximity to a lower dimensional face close to $\bx$ -- this provides points which give a good chance of containing $\bx$ in the convex hull reasonably quickly.  (The fact that the points $p_i^\bx$ are large in coordinate $i$ lets us view them as a diagonally dominant matrix, which we exploit to show that $\bx$ is likely to lie in their convex hull.)   Linearity of expectation will then show that the measure of the uncovered part of $\Omega(\be,\g)$ is very small, and Markov's inequality can then give a w.h.p statement as in the theorem.  In particular, although $\bx$ lying in the convex hull of $\set{\bq_1,\dots,\bq_N}$ is of course equivalent to $\bx$ lying in some simplex $S_x$ with vertices in $\set{\bq_1,\dots,\bq_N}$, it is perhaps surprising that we prove the theorem by actually identifying $S_x$, rather than, say, considering whether $\bx$ is separated from the convex hull by a hyperplane.

  \subsection{The exponential model}
A basic tool we use is the standard fact that the coordinate vector of a uniformly random point in the simplex $\Omega$ can be simply described using independent exponentials, as encapsulated in the first part of the following lemma.
\begin{lemma}\label{l.uniform}
  If we generate a random point $\bq\in \Omega$ by generating the coordinates $q_j$ as 
\begin{equation}\label{e.q}
q_{j}=\frac{E_{j}}{E_{1}+\dots+E_{d}},
\end{equation}
where the $E_i$'s are independent, mean 1 exponentials, then $\bq$ is uniform in $\Omega$.  Moreover, if we generate points $\bp_i$ $(i=1,\dots,d)$ by generating the coordinates as
\begin{align}\
  p_{i,j}&=\frac{\alpha E_{i,j}}{\sum_{k\neq i} E_{i,k}+\alpha E_{i,i}} \quad \mbox{for}\quad i\neq j\label {e.pdiff}\\
  p_{i,i}&=(1-\alpha)+\frac{\alpha^2 E_{i,i}}{\sum_{k\neq i} E_{i,k}+\alpha E_{i,i}}\label{e.psame},
\end{align}
where the $E_{i,j}$s are independent mean-1 exponentials, then each $\bp_i$ is uniform in the cap $C_i(\alpha)$.
\end{lemma}
\begin{proof}  The statement about $\bq$ is well-known and follows from the fact that the coordinate vector of a random point in $\Omega$ has the same distribution as the vector of $d$ gaps among $d-1$ independent uniforms in $[0,1]$, and that these gaps are distributed as exponentials with a conditioned sum (see e.g., \cite{Luc}, Ch 5, Theorems 2.1 and 2.2).
  
Consider now a point $\bp_i\in \Omega$ which is uniform except that we condition it to lie in $C_i(\alpha)$. Then for any Borel subset $B$ of $C_i(\alpha)$, we have
\[
\Pr\brac{\bp_i \in B} = \frac{\Vol(B \cap C_i(\alpha))}{\Vol(C_i(\alpha))} = \frac{\Vol(B)}{\Vol(C_i(\alpha))},
\]
so $\bp_i$ is uniform on $C_i(\alpha)$. Thus, in view of \eqref{e.cap} and \eqref{e.q}, the coordinates $p_{i,j}$ of $\bp_i$ are distributed as 
\begin{equation}\label{e.p}
p_{i,j}\sim \frac{E_{i,j}}{E_{i,1}+\dots+E_{i,d}}\quad\mbox{conditioned on}\quad E_{i,i}\geq (1-\alpha)\sum_{j=1}^d E_{i,d}.
\end{equation}
for independent mean-1 exponentials $E_{i,j}$. After solving for $E_{i,i}$, this conditioning is equivalent to conditioning on
\[
E_{i,i}\geq \frac{(1-\a)\sum_{j\neq i} E_{i,j}}{\a}.
\]
Note that for an exponential random variable $X$, the memoryless property implies that $X$ conditioned on $X > a$ has the same distribution as $X + a$. Thus, rather than the condition in \eqref{e.p}, thanks to the independence of $E_{i,i}$ and $\{E_{i,j}\}_{j: j\neq i}$, we could have instead replaced $E_{i,i}$ in that expression with a random variable $\hE_i$ generated as 
\[
\hE_i=\frac{(1-\a)\sum_{j\neq i}E_{i,j}}{\a}+E_{i,i},
\]
and \eqref{e.pdiff} and \eqref{e.psame} follow by substitution.
\end{proof}

We will also use the following result of Janson, which gives concentration for sums of exponentials.
\begin{lemma}[Janson \cite{Jan}]\label{l.janson}
Let $W_1,W_2,\ldots,W_m$ be independent exponentials with means $\frac{1}{a_i},i=1,2,\ldots,m$. Let $a_*=\min_{i=1}^ma_i$ and let $W=W_1+W_2+\cdots+W_m$ and $\m=\E(W)=\sum_{i=1}^m\frac{1}{a_i}$. Then, for any $\l\leq 1$,
\beq{Janson}{
\Pr(W\leq \l\m)\leq e^{-a_*\m(\l-1-\log\l)}.
}
\end{lemma}

\subsection{The large typical set}
Recall that our proof works by defining a large set of ``typical'' points in $\Omega$, and then showing that any such point is very unlikely to be still uncovered after exponentially many steps.

To define and work with the appropriate typical set, we will be interested in the magnitudes of the smallest coordinates of points $\bx$ in the set.  (Roughly speaking, the typical set $\Omega(\be,\g)$ defined below is one where none of smallest coordinates are much too small.)  For this purpose, we make the following definitions:
\begin{definition}\label{d.rank}
  Given a point $\bx\in \Omega$, $r_\bx(i)$ is the integer giving the ranking of $x_i$ among the coordinates $x_1,\dots,x_d$ of $\bx$, where ties are broken arbitrarily.  More precisely, $r_\bx:\{1,\dots,d\}\to \{1,\dots,d\}$ is any fixed bijection such that $r_\bx(i)\leq r_\bx(j)$ implies $x_i\leq x_j$.
\end{definition}
\begin{definition}\label{d.irank}
  Given a point $\bx\in \Omega$, $i_\bx$ is the integer $j\in \{1,\dots,d\}$ such that $i=r_\bx(j)$.
\end{definition}

In other words, if $(x_1^*,\ldots,x_d^*)$ is the nondecreasing rearrangement of $\bx = (x_1,\ldots,x_d)$, that is $x_1^* \leq \ldots \leq x_d^*$, then $(x_{i_\bx})_{i=1}^d = (x_i^*)_{i=1}^d$.

We now define our typical set as follows:
\begin{equation}\label{e.omeg}
\Omega(\be,\g)=\set{\bx\in \Omega:x_{i_\bx}\geq \frac{\e_i i}{d^2},1\leq i\leq \g d\text{ and }x_{i_\bx}\geq \frac{\g}{2d},i>\g d},
\end{equation}
where the coordinates of the vector $\be$ are defined in terms of a constant $\e>0$ and by 
\[\
\e_i=\begin{cases}e^{-\sqrt{d}},&1\leq i\leq \sqrt{d},\\\e,&i>\sqrt d.\end{cases}
\]

\begin{lemma}  \label{l.typicalvolume} For every $\g < 1$, there is a positive constant $c_\gamma$ such that for every $0< \e \leq \frac{1}{8} $ and $d$ large enough, we have
  \begin{equation}\label{e.typicalvolume}
\frac{\vol(\Omega(\be,\g))}{\vol(\Omega)}\geq 1-e^{-c_\g\sqrt{d}}.
  \end{equation}
\end{lemma}
\begin{proof}
Let $\bx$ be a random vector uniform on $\Omega$. In view of Lemma \ref{l.uniform} and \eqref{e.q}, the vector $(x_{i_\bx})_{i=1}^d = (x_i^*)_{i=1}^d$ of the order statistics of $\bx$ has the same distribution as the vector of the order statistics of i.i.d. mean one exponentials normalised by their sum. We recall the following classical result.
\begin{theorem}[Theorem 2.3, Chapter 5, \cite{Luc}]
Let $E_1, \ldots, E_n$ be independent mean one exponential random variables and let $E_{(1)} \leq E_{(2)} \leq \ldots \leq E_{(n)}$ be their order statistics, that is a nondecreasing rearrangement of the sequence $E_1,\ldots,E_n$. Then the vector $(E_{(1)},\ldots, E_{(n)})$ has the same distribution as the vector
\[
\left(\frac{E_1}{n}, \frac{E_1}{n} + \frac{E_2}{n-2},\ldots,\frac{E_1}{n}+\dots+\frac{E_n}{1}\right).
\]
\end{theorem}
\noindent
This gives that $(x_{i_\bx})_{i=1}^d$ has the same distribution as the vector
\beq{compsum}{
\left(\frac{E(d)+E(d-1)+\cdots+ E(d-i+1)}{\sum_{j=1}^d jE(j)}\right)_{i=1}^d,
}
where the $E(j)$'s are independent exponentials with rate $j$, that is the $jE(j)$ are independent mean one exponentials. Thus,
\begin{align*}
\frac{\vol(\Omega(\be,\g))}{\vol(\Omega)} &= \pp{\brac{x_{i_x} \geq \frac{\e_ii}{d^2},\;\forall 1 \leq i \leq \gamma d \ }\wedge \brac{x_{i_x} \geq \frac{\gamma}{2d},\; \forall\, \gamma d < i \leq d}} \\
&=\pp{\brac{x_{i_x} \geq \frac{\e_ii}{d^2},\;\forall 1 \leq i \leq \gamma d \ } \wedge  \brac{x_{(\lfloor \gamma d\rfloor + 1)_x} \geq \frac{\gamma}{2d}}} \\
&\geq 1 - \sum_{i \leq \gamma d} \pp{x_{i_x} \leq \frac{\e_ii}{d^2}} - \pp{x_{(\lfloor \gamma d\rfloor + 1)_x} \leq \frac{\gamma}{2d}}.
\end{align*}
We estimate these probabilities using Janson's inequality \eqref{Janson}. First define the event
\[
U = \left\{\sum_{j=1}^d jE(j) > \frac{8d}{5}\right\}.
\]
By \eqref{Janson}, applied with $W_j = jE(j)$, $a_1 = \ldots = a_d = 1$, $\mu = d$, $a_* = 1$, $\lambda = \frac{8}{5}$,
\[
\pp{U} \leq e^{-d(\frac{8}{5}-1-\log \frac{8}{5})} < e^{-d/10}.
\]
Now consider the events
\[
U_i = \left\{ E(d)+E(d-1)+\dots+E(d-i+1) \leq \frac{8\e_ii}{5d} \right\}.
\]
Set
\[
\m_i=\E(E(d)+E(d-1)+\cdots E(d-i+1)) = \sum_{j=d-i+1}^d\frac{1}{j}.
\]
Lower-bounding all the terms by the last one $\frac{1}{d}$, we have
\[
\m_i \geq \frac{i}{d}.
\]
By \eqref{Janson},
\[
\pp{U_i} \leq \exp\set{-\frac{(d-i+1)i(1.6\e_i-1-\log(1.6\e_i))}{d}}.
\]
Since $u-1-\log u > -\frac{1}{2}\log u$ for $u \leq 0.2$, we get for $i \leq \g d$, as long as $1.6\e \leq 0.2$,
\[
\pp{U_i} \leq (1.6\e_i)^{(1-\g)i/2}.
\]
Thus,
\[
\pp{x_{i_x} \leq \frac{\e_ii}{d^2}} \leq \pp{U_i} + \pp{U} \leq (1.6\e_i)^{(1-\g)i/2} + e^{-0.1d}
\]
and
\[
\sum_{i \leq \gamma d} \pp{x_{i_x} \leq \frac{\e_ii}{d^2}} \leq \sum_{i\leq \sqrt{d}} (1.6e^{-\sqrt{d}})^{(1-\g)i/2} + \sum_{\sqrt{d} < i \leq \g d} (1.6\e)^{(1-\g)i/2} + \g d e^{-0.1d} = e^{-\Omega(\sqrt{d})}.
\]
Similarly, for $i = \lfloor \gamma d\rfloor + 1$, we get $\mu_i \geq \frac{i}{d} \geq \gamma$, so 
\begin{align*}
\pp{x_{(\lfloor \gamma d\rfloor + 1)_x} \leq \frac{\gamma}{2d}} 
&\leq \pp{E(d)+\ldots+E(d-i+1) \leq 0.8\g} + \pp{U} \\
&\leq e^{-(d-i+1)\g (0.8-1-\log 0.8)} + e^{-0.1d} \\
&\leq e^{-0.02d(1-\g)\g} + e^{-d/10}.
\end{align*}
Putting these bounds together finishes the proof.
\end{proof}

\subsection{A lightly conditioned candidate simplex}
We now fix an arbitrary $\bx\in \Omega(\e,\g)$, and consider choosing a $\bp_i$ randomly from $C_i(\a)$, for some $i\in \{1,\dots,d\}$, using Lemma \ref{l.uniform}.  To use $\bp_i$ as the vertex of a candidate simplex to contain $\bx$, we hope to find that
\begin{equation}\label{e.smallip}
  p_{i,j_\bx}\leq \frac{\e_{j}j}{2d^2}\leq x_{j_\bx}/2,\quad\text{where}\quad j=1,2,\ldots,\g d
\end{equation}
runs over the smallest $\g d$ coordinates of $\bx$; recall that $j_\bx$ denotes the coordinate of the $j$th smallest component of $\bx$. Indeed, we will later argue that conditioning on this event for every $i$, the random points $\bp_1,\dots,\bp_d$ would have a reasonable chance of containing $\bx$ in their convex hull.  The following lemma shows that we can ensure that \eqref{e.smallip} is not too unlikely to be satisfied, without much conditioning on the random variables $E_{i,j_\bx}$ for $j>\g d$.

\begin{lemma}\label{l.goodpi}
 Let $\gamma \leq \frac{1}{6}$ and $2\e\g \leq 5\a$. Let $\bx\in \Omega(\be,\gamma)$ and let $\bp_i$ be chosen randomly from $C_i(\alpha)$ for some fixed $i\in \{1,\dots,d\}$, as in Lemma \ref{l.uniform}.  Then for the event
  \begin{equation}\label{e.theB}
  \cB_{i,\bx}=\bigg\{\sum_{\substack{ k\neq i\\r_\bx(k)>\g d}}{E_{i,k}}\geq \frac {4d}{5}\bigg\}
  \end{equation}
  and an event $\cA_{i,\bx}$ depending only on the $E_{i,j}$ for which $r_\bx(j)\leq \g d$ (and so independent of $\cB_{i,\bx}$), we have
  \begin{equation}\label{e.goodpiresults}
  \Pr\brac{\cA_{i,\bx}}\geq \frac{1}{d}\bfrac{\e\g}{5e\a}^{\g d}e^{-d}
  ,\quad \Pr\brac{\cB_{{i,\bx}}}\geq 1-e^{-10^{-4}d},
  \end{equation}
and
\[
\cC_{{i,\bx}}\supseteq \cB_{{i,\bx}}\cap  \cA_{i,\bx},
\]
where $\cC_{{i,\bx}}$ is the event that
\[
\forall (1\leq j\leq \g d,\, j_\bx\neq i) \  p_{i,j_\bx}\leq \frac{\e_{j} j}{2d^2}.
\]
\end{lemma}

\begin{proof}
 We have 
  \begin{align*}
    \cC_{{i,\bx}}=&\set{\frac{\a E_{i,j_\bx}}{\a E_{i,i}+\sum_{k\neq i}{E_{i,k}}}\leq \frac{\e_{j} j}{2d^2},\quad \forall \;1\leq j\leq \g d, j_\bx\neq i}\\ \supseteq
    & \set{\a E_{i,j_\bx}\leq \frac{\e_{j} {j}}{2d^2}\sum_{k\neq i,{j_\bx}}{E_{i,k}},\quad \forall \;1\leq j\leq {\g}d, j_\bx\neq i}\\
    \supseteq&\set{\a E_{i,{j_\bx}}\leq \frac{{2}\e_{j} {j}}{{5d}},\quad \forall \;1\leq j\leq {\g}d, j_\bx\neq i}
    \cap \set{\sum_{\substack{{{j}>\g d}\\ j_{\bx}\neq i}}{E_{i,j_{\bx}}}\geq \frac {4d}{5}},
  \end{align*}
  The second event in the last line is $\cB_{i,\bx}$, and we define $\cA_{i,\bx}$ to be the first event in the last line.  We have the claimed probability bound on $\cB_{i,\bx}$ from Lemma \ref{l.janson}. Indeed, for the mean $\mu = \E\sum_{\substack{{{j}>\g d}\\ j_{\bx}\neq i}}{E_{i,j_{\bx}}}$, we have $\mu \geq (1-\g)d$, so \eqref{Janson} gives
\begin{align*}
\pp{\sum_{\substack{{{j}>\g d}\\ j_{\bx}\neq i}}{E_{i,j_{\bx}}} \leq \frac{4d}{5} } &\leq \pp{\sum_{\substack{{{j}>\g d}\\ j_{\bx}\neq i}}{E_{i,j_{\bx}}} \leq \frac{4}{5(1-\g)}\mu } \\
&\leq \exp\left\{-\mu\left(\frac{4}{5(1-\g)} - 1 - \log\frac{4}{5(1-\g)}\right) \right\} \\
&\leq \exp\left\{-d(1-\g)\left(\frac{4}{5(1-\g)} - 1 - \log\frac{4}{5(1-\g)}\right) \right\}
\end{align*}
and for $\g \leq \frac{1}{6}$, we have $(1-\g)\left(\frac{4}{5(1-\g)} - 1 - \log\frac{4}{5(1-\g)}\right) > 10^{-4}$.

For $\cA_{\bx,i}$, we compute
  \begin{align*}
    \Pr(\cA_{i,\bx})=&\Pr\brac{\a E_{i,{j_\bx}}\leq \frac{2\e_{j} {j}}{5d},\quad \forall \;1\leq j\leq d, j_\bx\neq i}\\
    =&\prod_{\substack{1\leq j\leq \g d\\j_\bx\neq i}}\Pr\brac{E_{i,{j_\bx}}\leq \frac{2\e_{j} {j}}{5\a d}}\\
    {=} & \prod_{\substack{1\leq j\leq \g d\\j_\bx\neq i}}\brac{1-\exp\brac{-\frac{2\e_{j}j}{5\a d}}}\\
        \geq& \prod_{\substack{1\leq j\leq \g d\\j_\bx\neq i}}\frac{\e_{j}j}{5\a d}
  \end{align*}
for $2\e\g\leq 5\a$, since $1-e^{-b}\geq b-\frac {b^2} 2\geq \frac b 2$ for $b\leq1$.
Thus we have
\[\pushQED{\qed}
  \Pr(\cA_{i,\bx})
\geq {\frac{\lfloor \g d\rfloor!}{(5\a d)^{\g d}}\prod\limits_{\substack{1\leq j\leq \g d\\j_\bx\neq i}}\e_{j}}
\geq \frac{(\g d/e)^{\g d-1}}{(5\a d)^{\g d}}\cdot \brac{e^{-\sqrt d}}^{\sqrt d}\cdot \e^{\g d}\geq \frac{1}{d}\bfrac{\e\g}{5e\a}^{\g d}e^{-d}.\qedhere
\]
\end{proof}

As a consequence of Lemma \ref{l.goodpi}, we will have that if we sample exponentially many points in $C_i(\a)$, we will with probability at least $1-e^{-d}$ have at least one one point $\bp_i$ for which the corresponding event $\cA_{i,\bx}$ occurs. In particular, we will with probability at least $1-de^{-d}$ have one such point $\bp_i$ for each $i=1,\dots,d$.  Furthermore, with probability $1-de^{-10^{-4} d}$, we have that all the corresponding events $\cB_{i,\bx}$ occur. These points $\bp_1,\dots,\bp_d$ form the vertices of a candidate simplex; note that the Lemma gives us that these points $\bp_i$ satisfy $p_{i,j_\bx}\leq \frac{\e_{j_\bx}j_\bx}{2d^2}$ for all $1\leq j_\bx \leq \g d$, $j\neq i$. In the next section, we show that they are not too unlikely to contain the fixed vertex $\bx\in \Omega(\be,\g)$.  In particular, this will mean that after collecting exponentially many such simplices (in time $\text{exponential}(d)\cdot N=\text{exponential}(d)$), the probability that $\bx$ is not covered by any such simplex will be exponentially small.

\subsection{Enclosing a fixed $\bx\in \Omega(\e,\g)$}

In this section we show that for any fixed $\bx\in \Omega(\e,\g)$, it is only exponentially unlikely to be contained in a simplex whose vertices $\bp_i,\,i=1,2,\ldots,d$ are each chosen randomly from the corresponding set $C_i(\a)$.

In particular, our goal in this section is to prove:
\begin{lemma}\label{l.contain}  Let $\gamma \leq \frac{1}{6}$ and $2\e\g \leq 5\a$. Fix $\bx\in \Omega(\e,\g)$ and suppose that for each $i=1,\dots,d$, the point $\bp_i$ is chosen randomly from $C_i(\a)$. Let $\cA_{i,\bx}$ be the events from Lemma \ref{l.goodpi} and let $\cA_\bx = \bigcap_{i=d}^n \cA_{i,\bx}$. Then
  \[
 \pp{ \bx \in \conv\{\bp_1,\ldots,\bp_d\} \ \big| \ \cA_\bx} \geq \delta_d,
  \]
  where
  \[
\delta_d = \left(1 - \frac{5\a}{\g}\right)^{(1-\g)d} - de^{-10^{-4}d}.
  \]
\end{lemma}

We define the matrix $\cP=p_{j,i} $ whose rows are the random points $\bp_i$, and write $\cP=\cD+\cR$ where $\cD$ is the diagonal of $\cP$, and $\cM=\cD^{-1}\cR$.

We will apply Gershgorin's Circle Theorem to this matrix $\cM$.
\begin{theorem}[Gershgorin]
  Suppose $\cM=[m_{ij}]$ is a real or complex $d\times d$ matrix where for each $i=1,\dots,d$, $R_i=\sum_{j\neq i} |m_{ij}|$ is the sum of the absolute values of the non-diagonal entries of the $i$th row, and the $i$th \emph{Gershgorin disc} $D_i$ is the disc of radius $R_i$ centered at $m_{ii}$.  Then every eigenvalue of $\cM$ lies in one of the Gershgorin discs (and, applying this to $\cM^T$, the same applies where we define the Gershgorin discs with respect to the columns).\qed
\end{theorem}

In particular, we use it to prove the following statement.
\begin{lemma}
   We have that
  \begin{equation}\label{e.assum}
    \cP^{-1}=\left(\sum_{k=0}^\infty(-1)^k\cM^k\right)\cD^{-1}.
  \end{equation}
\end{lemma}
\begin{proof}
Observe that if the sum in \eqref{e.assum} converges, then we can write
  \begin{equation}
    \cP^{-1}=(\cI+\cD^{-1}\cR)^{-1}\cD^{-1}=\left(\sum_{k=0}^\infty(-1)^k(\cD^{-1}\cR)^k\right)\cD^{-1}=\left(\sum_{k=0}^\infty(-1)^k\cM^k\right)\cD^{-1}.
  \end{equation}
  Thus it remains just to prove that the sum converges.  Recall first from the definition of $C_i(\a)$ in \eqref{e.cap} that diagonal entries of $\cP$ are all at least $1-\a$, while the sum of each row is 1.  In particular, Gershgorin's Circle Theorem implies the eigenvalues of $\cM$ have absolute value at most $\frac{\a}{1-\a}$, which is less than 1 assuming $\a<\frac 1 2$.  

  Next we argue that $\cM$ is a.s.~diagonalizeable.  This is the case if the discriminant of the characteristic polynomial of $\cM$ is nonzero.  This discriminant is a polynomial expression involving only products of the off-diagonal entries of $\cM$; in particular, it is nonzero with probability 1.

  Thus finally we write $\cM=Q\Lambda Q^{-1}$ and $\cM^k=Q\Lambda^k Q^{-1}$.  This converges exponentially fast, confirming convergence of the sum and thus the lemma.
\end{proof}
We are now ready to prove Lemma \ref{l.contain}.
\begin{proof}[Proof of Lemma \ref{l.contain}]
  As the $\bp_i$ lie in general position, we can always write the given $\bx$ (uniquely) as a linear combination
  \[
  \bx=\lambda_1 \bp_1+\cdots+\lambda_d\bp_d;
  \]
  our goal is to show that given $\bigcap_i \cA_{i,\bx}$, there is probability at least $c^d$ for some $c>0$ that the $\lambda_i$ are all nonnegative.  Observe that these coefficients are determined as
  \[\bl=\bx\cP^{-1}.\]

  From \eqref{e.assum}, we can write
  \begin{equation}
    \bx\cP^{-1}=\bx\left(\sum_{k=0}^\infty (-1)^k \cM^k\right)\cD^{-1}=\bx\big(I-\cM\big)\left(\sum_{k=0}^\infty \cM^{2k}\right)\cD^{-1},
  \end{equation}
  which is nonnegative so long as $\bx(I-\cM)$ is, since $\cM$ has only nonnegative entries.

Note that the $j$th coordinate $y_j$ of the product $\by=\bx(I-\cM)$ is given by
  \begin{equation}\label{e.ycoor}
  y_j=x_j-\sum_{i: i\neq j} D_{i,i}^{-1}p_{i,j}x_i.
  \end{equation}
  
Thus
\[
\pp{ \bx \in \conv\{\bp_1,\ldots,\bp_d\} \ \big| \ \cA_\bx} \geq \pp{y_j \geq 0:\, \forall j \leq d  \big| \ \cA_\bx}.
\]
  
Recall from Lemma \ref{l.goodpi} the events $\cB_{i,\bx}$ which are all independent of the events $\cA_{i,\bx}$. Let $\cB_{\bx}=\bigcap_{i=1}^d \cB_{i,\bx}$. Each of the values of $j$ corresponding to small coordinates of $\bx$---that is the $j$ for which $r_\bx(j)\leq \g d$---must satisfy $y_j\geq 0$ if $\cA_\bx\cap \cB_\bx$ occurs. Indeed, from Lemma \ref{l.goodpi}, we know that for all $i\neq j$ we have $p_{i,j}\leq x_j/2$ in this case, and so in particular we have that 
 \beq{yjge0}{
  y_j=x_j-\sum_{i\neq j} D_{i,i}^{-1}p_{i,j}x_i\geq x_j-\frac{x_j}{2(1-\a)}\sum_{i\neq j}x_i\geq x_j-\frac{x_j}{2(1-\a)}>0
 }
(by \eqref{e.psame}, we have $D_{i,i} \geq 1-\a$). This shows that
\begin{align*}
\pp{ y_j \geq 0 :\forall j \leq d \big| \ \cA_\bx} &\geq \pp{(y_j \geq 0: \forall j \leq d)  \cap \cB_\bx \ \big| \ \cA_\bx} \\
&= \pp{(y_j \geq 0: \forall j \leq d \text{ s.t. } r_\bx(j) > \g d ) \cap \cB_\bx \ \big| \ \cA_\bx}.
\end{align*}

  It therefore remains to handle the case of $r_\bx(j)>\g d$. On the event $\cB_\bx$, for $i \neq j$, we have $p_{i,j} \leq \frac{5\a}{4d} E_{i,j}$ (recall \eqref{e.pdiff}), thus, bounding $D_{i,i}^{-1} \leq \frac{1}{1-\a} \leq 2$ (see \eqref{e.consts}) and using that $x_j \geq \frac{\g}{2d}$ for $j$ such that $r_\bx(j) > \g d$, we obtain (by using the first equality in \eqref{yjge0}) that
 \mults{
\pp{(y_j \geq 0: \forall j \leq d, r_\bx(j) > \g d ) \cap \cB_\bx \ \big| \ \cA_\bx}\\
\geq \pp{ \brac{ \sum_{i: i \neq j} E_{i,j}x_i \leq \frac{\g}{5\a}:  \forall j \leq d \text{ s.t. }r_\bx(j) > \g d }\cap \cB_\bx \ \big| \ \cA_\bx}.
}
By a simple inequality $\pp{A \cap B} \geq \pp{A} - \pp{B^c}$, 
\[
\pp{y_j \geq 0 :\forall j \leq d \ \ \big| \ \cA_\bx} \geq \pp{\sum_{i: i \neq j} E_{i,j}x_i \leq \frac{\g}{5\a}: \forall j \leq d \text{ s.t. }r_\bx(j) > \g d\ \big| \ \cA_\bx}  -\pp{ \cB_\bx^c \ \big| \ \cA_\bx}
\]
The fact that the $E_{i,j}$ for $j$ with $r_\bx(j) > \g d$ are not conditioned by $\cA_\bx$, Markov's inequality and independence yields
\mults{
\pp{ \sum_{i: i \neq j} E_{i,j}x_i \leq \frac{\g}{5\a}:\;\forall j \leq d \text{ s.t. } r_\bx(j) > \g d \ \big| \ \cA_\bx}\\ = \pp{\sum_{i: i \neq j} E_{i,j}x_i \leq \frac{\g}{5\a}:\;\forall j \leq d \text{ s.t. } r_\bx(j) > \g d }
\geq \left(1 - \frac{5\a}{\g}\right)^{(1-\g)d}.
}
The independence of $\cB_\bx$ and $\cA_\bx$, a simple union bound and \eqref{e.goodpiresults} yields
\[
\pp{ \cB_\bx^c \ \big| \ \cA_\bx} = \pp{\cB_\bx^c} \leq de^{-10^{-4}d}.
\]
Altogether,
\beq{ineq}{
\pp{ \bx \in \conv\{\bp_1,\ldots,\bp_d\} \ \big| \ \cA_\bx} \geq \left(1 - \frac{5\a}{\g}\right)^{(1-\g)d} - de^{-10^{-4}d}.
}
\end{proof}

  
    \subsection{Covering most of the simplex in exponentially many steps}
    \label{s.mainproof}
We are now ready to combine the ingredients to prove the main theorem.

\begin{proof}[Proof of Theorem \ref{t.coveromega}]
Recall that we draw $N$ random points $\bq_1,\bq_2,\dots,\bq_N$ independently and uniformly from the simplex $\Omega$ and $Q_N$ denotes their convex hull. First, note that by Fubini's theorem, we have
\begin{equation}\label{e.Evol}
\E \Vol(Q_N) = \E \int_{\Omega} \1_{\{x \in Q_N\}} dx = \int_{\Omega} \pp{x \in Q_N} dx \geq \int_{\Omega(\be,\g)} \pp{x \in Q_N} dx,
\end{equation}
where $\Omega(\be,\g)$ is the typical set defined in \eqref{e.omeg}. Fix $x \in  \Omega(\be,\g)$. By Lemma \ref{l.contain}, we will have a good lower bound on $\pp{x \in Q_N}$ provided we know that among the $\bq_i$ there are $d$ points, one from each cap $C_i(\a)$ which moreover fulfill the events $\cA_\bx$. To use that, we condition on all possibilities for the $\bq_i$ and then argue that the majority of the possibilities are \emph{good}, provided $N$ is large enough. Formally, given two sequences $l=(l_1,\ldots,l_N) \in \{0,1,\ldots,d\}^N$ and $\theta = (\theta_1,\ldots,\theta_N) \in \{0,1\}^N$, we define the event
\begin{align*}
\cE_{l,\th} = \bigcap_{j \leq N} \Bigg\{\text{if $l_j = 0$, then } q_j \notin \bigcup_{i=1}^d C_i(\a); \ &\text{if $l_j > 0$, then } q_j \in C_{l_j}(\a) \\
&\text{ and } q_j \text{ satisfies  $\cA_{l_j,\bx}$ if and only if $\theta_j = 1$} \Bigg\}
\end{align*}
which tells us which among the points $\bq_j$ fall in the caps and among those which satisfy $\cA_{i,\bx}$. Let $\textsf{Good}$ be the set of those pairs of sequences $(l,\th)$ for which there are $1 \leq j_1 < \ldots < j_d \leq N$ such that $\{l_{j_1},\ldots,l_{j_N}\} = \{1,\ldots,d\}$ and $\theta_{j_1}=\ldots=\theta_{j_d} = 1$. Then,
\[
\pp{x \in Q_N} = \sum_{l, \th} \pp{x \in Q_N \ | \ \cE_{l,\th}}\pp{\cE_{l,\th}} \geq \sum_{(l, \th) \in \textsf{Good}} \pp{x \in Q_N \ | \ \cE_{l,\th}}\pp{\cE_{l,\th}}.
\]
For $(l,\th) \in \textsf{Good}$, by Lemma \ref{l.contain}, we have $\pp{x \in Q_N \ | \ \cE_{l,\th}} \geq \delta_d$, so it remains to estimate the sum $\sum_{(l, \th) \in \textsf{Good}} \pp{\cE_{l,\th}}$. Fix $i \in \{1,\ldots, d\}$ and let $S_i$ be the number of points among the $\bq_j$ which are in $C_i(\a)$ and satisfy $\cA_{i,\bx}$. We have,
\[
\sum_{(l, \th) \in \textsf{Good}} \pp{\cE_{l,\th}} = \pp{S_i > 0:\;\forall i \leq d }
\]
By independence,
\[
\pp{S_i = 0} = \Big(1 - \pp{\cA_{i,\bx}}\pp{\bq_1 \in C_i(\a)} \Big)^N,
\]
where $\pp{\cA_{i,\bx}}$ is taken with respect to the uniform probability on $C_i(\a)$. Therefore, by \eqref{e.goodpiresults} and a union bound,
\[
\sum_{(l, \th) \in \textsf{Good}} \pp{\cE_{l,\th}} \geq 1 - d\cdot \Bigg(1 - \frac{1}{d}\bfrac{\e\g}{5e\a}^{\g d}e^{-d}\a^{d-1} \Bigg)^N.
\]
Thus, 
\beq{another1}{
\pp{x \in Q_N} \geq \delta_d\cdot\brac{1 - d\cdot \exp\left\{-N\cdot\frac{1}{d}\bfrac{\e\g}{5e\a}^{\g d}e^{-d}\a^{d-1}\right\}}.
}
Set $\g = \frac{1}{6}$ and then choose $\a$ to be a small enough constant such that
\beq{ineq1}{
\delta_d = \left(1 - \frac{5\a}{\g}\right)^{(1-\g)d} - de^{-10^{-4}d} > e^{-10^{-4}d}.
}
Choose $\e \leq \frac18$ (allowing the use of Lemma \ref{l.typicalvolume} later) such that $2\e\g\leq 5\a$ (allowing the use of Lemma \ref{l.goodpi}). Then we take $N = C_1^d$ with $C_1$ large enough so that the exponential term in \eqref{another1} satisfies
\[
d\cdot \exp\left\{-N\cdot\frac{1}{d}\bfrac{\e\g}{5e\a}^{\g d}e^{-d}\a^{d-1}\right\}\leq\frac12.
\]
We then have
\[
\pp{x \in Q_N} \geq \frac12e^{-10^{-4}d}.
\]
Then, by independence, we get
\[
\pp{x \in Q_{2^dN}} \geq 1 - \pp{x \notin Q_N}^{C_2^d} \geq 1 - \left(1-\frac12e^{-10^{-4}d}\right)^{2^d} \geq 1 - \exp\left\{-\frac12e^{-10^{-4}d}\cdot 2^d\right\} > 1 - e^{-d}.
\]
Finally, thanks to \eqref{e.Evol} and Lemma \ref{l.typicalvolume},
\begin{align*}
\E \Vol(Q_{C_1^dN}) &\geq \Vol(\Omega(\be,\g))(1-e^{-d}) \geq 1 - e^{-c_0\sqrt{d}},
\end{align*}
for a  positive universal constant $c_0$.
\end{proof}

\begin{Remark}\label{rem:consts}
All of these inequalities hold with
\begin{equation}\label{e.consts}
\g = \frac{1}{6}, \quad \a = \frac{3}{100}, \quad \e = \frac{1}{8}, \quad C_1 = 150
\end{equation}
(provided $d$ is large enough). Moreover, for the constant $c_\g$ in Lemma \ref{l.typicalvolume}, we can take $c_\g = \frac38$. These justify Remark \ref{rem:consts-final}.
\end{Remark}

\section{Proof of the lower bound: Theorem \ref{t.lower}}

Since the quantity $\frac{1}{\Vol(\Omega_d)}\Vol(Q_N)$ is affine invariant, we can work with the standard orthogonal simplex $S_d$ in $\Re^d$ instead of $\Omega_d$,
\[
S_d = \{x \in \Re^d, \ x_1,\ldots, x_d \geq 0, \sum_{i=1}^d x_i \leq 1\},
\]
which will be more convenient here.
The following fundamental lemma from \cite{DFM} is a starting point.

\begin{lemma}[\cite{DFM}]\label{lm:central}
Suppose $\bq_1, \bq_2,\ldots$ are i.i.d. copies of a continuous random vector $\bq$ in $\Re^d$. Define a random polytope $Q_N = \text{conv}\{\bq_1,\ldots,\bq_N\}$ and consider the function $\xi=\xi_\bq$ defined by 
\begin{equation}\label{eq:def-q}
\xi(x) = \inf\{\pp{\bq \in H}, \ H \text{ half-space containing $x$}\}, \qquad x \in \Re^d.
\end{equation}
Then for every subset $A$ of $\Re^d$, we have
\begin{equation}\label{eq:upp-bd}
\E\Vol(Q_N) \leq \Vol(A) + N\cdot \left(\sup_{A^c} \xi\right) \cdot \Vol(A^c \cap \{x \in \Re^d, \ \xi(x) > 0\})
\end{equation}
and
\begin{equation}\label{eq:low-bd}
\E\Vol(Q_N) \geq \vol(A)\left(1 - 2\binom{N}{d}\left(1 - \inf_A \xi\right)^{N-d}\right).
\end{equation}
\end{lemma}

We will only need the first part of Lemma \ref{lm:central}, that is \eqref{eq:upp-bd}, which will be applied to sets of the form $A = \{x \in \Re^d, \ \xi(x) > \lambda\}$, the (convex) level sets of the function $\xi$. To get an upper bound on the volume of such sets, we shall use a standard lemma concerning the Legendre transform $\Lambda_\bq^\star$ of the log-moment generating function $\Lambda_\bq$ of $\bq$,
\[
\Lambda_\bq(x) = \log \E e^{\scal{\bq}{x}} \qquad \text{and} \qquad \Lambda_\bq^\star(x) = \sup_{\theta \in \Re^d} \left\{\scal{\theta}{x} - \Lambda_\bq(\theta)\right\}.
\]

\begin{lemma}
For every $\alpha > 0$, we have
\begin{equation}\label{eq:q-subset-Bt}
\{x \in \Re^d, \ \xi_\bq(x) > e^{-\alpha}\} \subset \{x \in \Re^d, \ \Lambda_\bq^\star(x) < \alpha\}.
\end{equation}
\end{lemma}
\begin{proof}
Plainly, for the infimum in the definition \eqref{eq:def-q} of $\xi_\bq(x)$, it is enough to take half-spaces for which $x$ is on the boundary, that is
\begin{equation}\label{eq:fromula-q}
\xi_\bq(x) = \inf_{\theta \in \Re^d} \pp{ \scal{\bq-x}{\theta} \geq 0},
\end{equation}
where $\scal{u}{v} = \sum_i u_iv_i$ is the standard scalar product in $\Re^d$. By Chebyshev's inequality for the exponential function,
\[
\pp{\scal{\bq-x}{\theta}\geq 0} \leq e^{-\scal{\theta}{x}}\E e^{\scal{\theta}{\bq}}.
\]
Consequently, $\xi_\bq(x) \leq e^{-\Lambda_\bq^*(x)}$.
\end{proof}

The next lemma is a crucial bound on the moment generating function $\Lambda_\bq$ for $\bq$ uniform on the simplex $S_d$. 

\begin{lemma}\label{lm:mom-gen}
Let $\bq$ be a random vector uniform on $S_d$. For every $\theta \in (-\infty,d)^d$ and $d \geq 7$, we have
\[
\E e^{\scal{\theta}{\bq}} \leq d\prod_{i=1}^d \frac{1}{1-\theta_i/d}.
\]
\end{lemma}
\begin{proof}
We have,
\[
\E e^{\scal{\theta}{\bq}} = \frac{1}{\Vol(S_d)}\int_{S_d} e^{\sum \theta_ix_i} dx = d!\int_{\{x \in (0,\infty)^d, \ \sum x_i \leq 1\}}e^{\sum \theta_ix_i} dx.
\]
A change of variables $x_i = y_i/d$ and a simple pointwise estimate $1 \leq e^{d-\sum y_i}$ valid on the domain of the integration yield
\begin{align*}
\E e^{\scal{\theta}{\bq}} &\leq d!d^{-d}\int_{\{y \in (0,\infty)^d, \ \sum y_i \leq d\}}e^{\sum \theta_iy_i/d}e^{d-\sum y_i} dy \\
&\leq d!d^{-d}e^d\int_{(0,\infty)^d}e^{\sum -(1-\theta_i/d)y_i} dy \\
&= (d!d^{-d}e^d)\prod_{i=1}^d \frac{1}{1-\theta_i/d}.
\end{align*}
Finally, $d! < \sqrt{2\pi d}d^de^{-d}e^{\frac{1}{12d}}$. For $d \geq 7$, we have $\sqrt{2\pi d}e^{\frac{1}{12d}} \leq d$.
\end{proof}

\begin{proof}[Proof of Theorem \ref{t.lower}.]
Fix $\varepsilon > 0$. Let $N \leq e^{( \gamma-\varepsilon) d}$ and $\alpha = \gamma - \varepsilon/2$. Let $\xi$ be the function from \eqref{eq:def-q} defined for a random vector $\bq$ uniformly distributed on $S_d$. Setting $Q_N = \text{conv}\{\bq_1,\dots,\bq_N\}$, where $\bq_1, \bq_2,\dots$ are i.i.d. copies of $\bq$ and using \eqref{eq:upp-bd} with $A = \{x \in S_d, \ q(x) > e^{-\alpha d} \}$, we get
\[
\frac{\E \vol(Q_N)}{\vol(S_d)} \leq \frac{\vol(A)}{\vol(S_d)} + e^{-\varepsilon d/2}
\]
By \eqref{eq:q-subset-Bt},
\[
\frac{\vol(A)}{\vol(S_d)} = \pp{\xi(\bq) > e^{-\alpha d}} \leq \pp{\Lambda^\star(\bq) < \alpha d}.
\]
By Lemma \ref{lm:mom-gen}, we obtain
\begin{align*}
\Lambda^\star(x) \geq \sup_{\theta \in (-\infty,d)^d} \left\{\scal{\theta}{x} - \log\prod_{i=1}^d \frac{1}{1-\theta_i/d} - \log d \right\} &= -\log d+\sum_{i=1}^d\sup_{\theta_i <d} \left\{\theta_ix_i + \log\left(1-\frac{\theta_i}{d}\right) \right\} \\
&= -\log d+\sum_{i=1}^d \psi(x_id),
\end{align*}
where
\[
\psi(t) = t-1-\log t, \qquad t > 0.
\]
As a result, for $d$ large enough,
\[
\pp{\Lambda^\star(\bq) < \alpha d} \leq \pp{\frac{1}{d}\sum_{i=1}^d \psi(q_id) < \alpha + \frac{\log d}{d}} \leq \pp{\frac{1}{d}\sum_{i=1}^d \psi(q_id) < \gamma - \frac{\varepsilon}{4}}
\]
(here the $\bq=(q_1,\ldots,q_d)$, so the $q_i$ are the components of $\bq$). To finish the proof, it remains to argue that the right hand side is $o(1)$. We use the fact that $\bq$ has the same distribution as the vector $(\frac{Y_1}{Z},\ldots, \frac{Y_d}{Z})$, where $Z = Y_1 + \dots+Y_d + W$ and $Y_1, \ldots, Y_d, W$ are i.i.d. exponential random variables with parameter $1$ (which can be deduced e.g. from \eqref{e.q} of Lemma \ref{l.uniform}). For $\delta \in (0,1)$ (to be chosen later), we write
\begin{align*}
\pp{\frac{1}{d}\sum_{i=1}^d \psi(q_id) < \gamma - \frac{\varepsilon}{4}} &= \pp{\frac{1}{d}\sum_{i=1}^d \psi\left(d\frac{Y_i}{Z}\right) < \gamma - \frac{\varepsilon}{4}}\\
&\leq \pp{\frac{1}{d}\sum_{i=1}^d \psi\left(d\frac{Y_i}{Z}\right) < \gamma - \frac{\varepsilon}{4}, \ Z \in ((1-\delta)d,(1+\delta)d)}\\
&\quad+ \pp{Z < (1-\delta)d} + \pp{Z > (1+\delta)d}.
\end{align*}
The last two probabilities are exponentially small (which can be argued in a number of ways, e.g. using Lemma \ref{l.janson}, Bernstein's inequality, or estimates for the incomplete gamma function). To handle the first probability, we decompose $\psi$ as follows
\[
\psi(t) = \psi_1(t) + \psi_2(t),
\]
where $\psi_1(t) = \psi(t)\1_{(0,1]}(t)$ is nonincreasing and $\psi_2(t) = \psi(t)\1_{(1,\infty)}(t)$ is nondecreasing. Having this monotonicity, if $Z \in ((1-\delta)d,(1+\delta)d)$, we get $\psi\left(d\frac{Y_i}{Z}\right) \geq  \psi_1\left(\frac{Y_i}{1-\delta}\right) + \psi_2\left(\frac{Y_i}{1+\delta}\right)$. Thus, setting
\[
f(t) = \psi_1\left(\frac{t}{1-\delta}\right) + \psi_2\left(\frac{t}{1+\delta}\right),
\]
we obtain
\[
\pp{\frac{1}{d}\sum_{i=1}^d \psi\left(d\frac{Y_i}{Z}\right) < \gamma - \frac{\varepsilon}{4}, \ Z \in ((1-\delta)d,(1+\delta)d)} \leq \pp{\frac{1}{d}\sum_{i=1}^d f(Y_i) < \gamma - \frac{\varepsilon}{4}}.
\]
It remains to find the mean of $f(Y_1)$ and use the law of large numbers. We have,
\begin{align*}
\E f(Y_1) = \int_0^\infty f(t)e^{-t} dt &= \int_0^{1-\delta} \psi\left(\frac{t}{1-\delta}\right)e^{-t} dt + \int_{1+\delta}^\infty \psi\left(\frac{t}{1+\delta}\right)e^{-t} dt \\
&= (1-\delta)\int_0^1 \psi(t) e^{-t}e^{\delta t} dt + (1+\delta)\int_1^\infty \psi(t)e^{-t}e^{-\delta t} dt \\
&\geq (1-\delta)\int_0^1 \psi(t) e^{-t} dt + \int_1^\infty \psi(t)e^{-t}(1-\delta t) dt \\
&= \int_0^\infty \psi(t)e^{-t} dt - \delta\left(\int_0^1 \psi(t)e^{-t}dt + \int_1^\infty \psi(t)te^{-t} dt\right).
\end{align*}
Since
\[
\int_0^\infty \psi(t)e^{-t} dt = -\int_0^\infty e^{-t}\log t dt = \gamma
\]
(which was derived by Euler -- see (2.2.8) in the survey \cite{Lag})
and
\[
\int_0^1 \psi(t)e^{-t} dt + \int_1^\infty \psi(t)te^{-t} dt < 1,
\]
we can conclude that
\[
\E f(Y_1) > \gamma - \delta.
\]
Choosing, say $\delta = \frac{\varepsilon}{8}$, we thus get
\[
\pp{\frac{1}{d}\sum_{i=1}^d f(Y_i) < \gamma - \frac{\varepsilon}{4}} \leq \pp{\frac{1}{d}\sum_{i=1}^d f(Y_i) < \E f(Y_1) - \frac{\varepsilon}{8}}
\]
and by the (weak) law of large numbers, the right hand side converges to $0$ as $d \to \infty$.
\end{proof}

\section*{Acknowledgments} 
The authors are grateful to the anonymous reviewer for useful comments.

\bibliographystyle{amsplain}


\begin{dajauthors}
\begin{authorinfo}[af]
  Alan Frieze\\
  Department of Mathematical Sciences\\
  Carnegie Mellon University\\
  Pittsburgh PA 15213, USA\\
  alan\imageat{}random\imagedot{}math\imagedot{}cmu\imagedot{}edu \\
  \url{https://www.math.cmu.edu/~af1p/}
\end{authorinfo}
\begin{authorinfo}[wp]
  Wesley Pegden\\
    Department of Mathematical Sciences\\
  Carnegie Mellon University\\
  Pittsburgh PA 15213, USA\\
  wes\imageat{}math\imagedot{}cmu\imagedot{}edu \\
  \url{http://math.cmu.edu/~wes/}
\end{authorinfo}
\begin{authorinfo}[t]
  Tomasz Tkocz\\
    Department of Mathematical Sciences\\
  Carnegie Mellon University\\
  Pittsburgh PA 15213, USA\\
  ttkocz\imageat{}math\imagedot{}cmu\imagedot{}edu \\
  \url{http://math.cmu.edu/~ttkocz/}
\end{authorinfo}
\end{dajauthors}

\end{document}